\documentclass[reqno,11pt]{amsart}
\usepackage{amscd,amssymb,verbatim}
\usepackage{enumitem}
\usepackage{marginnote}
\usepackage[nobysame]{amsrefs}
\usepackage{hyperref}
\usepackage[scr=rsfso]{mathalfa}
\usepackage{mathrsfs}  \usepackage{bbm}
\usepackage{xcolor}
\usepackage{slashed}
\usepackage{graphicx}
\usepackage{stmaryrd}

\numberwithin{equation}{section}

\theoremstyle{plain}
\newtheorem{Theorem}{Theorem}[section]
\newtheorem{Lemma}[Theorem]{Lemma}
\newtheorem{Corollary}[Theorem]{Corollary}
\newtheorem{Proposition}[Theorem]{Proposition}

\newtheorem{Assumption}[Theorem]{Assumption}

\theoremstyle{definition}
\newtheorem{Definition}[Theorem]{Definition}

\theoremstyle{remark}

\newcommand\alp{\alpha}         
\newcommand\bet{\beta}
\newcommand\gam{\gamma}         \newcommand\Gam{\Gamma}
         
\newcommand\eps{\varepsilon}

\newcommand\iot{\iota}

\newcommand\lam{\lambda}                \newcommand\Lam{\Lambda}

\newcommand\ome{\omega}         \newcommand\Ome{\Omega}

\newcommand\calE{{\mathcal{E}}}
\newcommand\calF{{\mathcal{F}}}

\newcommand\calH{{\mathcal{H}}}

\newcommand\calL{{\mathcal{L}}}

\newcommand\calU{{\mathcal{U}}}

\newcommand\bfu{{\mathbf u}}            
\newcommand\bfv{{\mathbf v}}

\newcommand\RR{\mathbb{R}}

\newcommand\CC{\mathbb{C}}

 \newcommand\grg{{\mathfrak{g}}}

\newcommand\nek{,\ldots,}
\newcommand\sdp{\times \hskip -0.3em {\raise 0.3ex
\hbox{$\scriptscriptstyle |$}}} 

\newcommand\End{\operatorname{End\,}}

\newcommand\IM{\operatorname{Im}}

\newcommand\Irr{\operatorname{Irr\, }}

\newcommand\Ker{\operatorname{Ker}}

\newcommand\Vol{\operatorname{Vol}}

\newcommand\oq{{\overline{q}}}

\newcommand\tilh{{\tilde{h}}}

\newcommand{\lap}{\bar\square_t}

\renewcommand{\>}{\rangle}
\newcommand{\<}{\langle}

\newcommand{\h}[1]{\calH^{#1}(M_0,L_0)}

\theoremstyle{plain}

\theoremstyle{remark}

\newtheorem{Remark}[subsection]{Remark}

\newcommand{\refe}[1]{\eqref{E:#1}}

\renewcommand{\d}{\text{\( \partial\)}}
\newcommand{\p}{\bar{\d}}
\renewcommand{\b}{\bullet}

\newcommand{\n}{\nabla}

\newcommand{\E}{\calE}

\renewcommand{\L}{\calL}

\newcommand{\ah}{^{0,1}}

\newcommand{\nLC}{\n^{\text{LC}}}

\renewcommand{\v}{\mathbf{v}} 
\renewcommand{\u}{\mathbf{u}}

\newcommand{\ka}{K\"ahler }

\newcommand{\bg}{{\operatorname{bg}}}

\begin{document}
\title{Background Cohomology and Symplectic Reduction}
\author[Maxim Braverman]{Maxim Braverman}
\address{Department of Mathematics\\
        Northeastern University   \\
        Boston, MA 02115 \\
        USA
         }
\email{m.braverman@northeastern.edu}
      

\begin{abstract}
We consider a Hamiltonian action of a compact Lie group $G$ on a complete \ka manifold $M$ with a proper moment map. In a previous paper, we defined a regularized version of the Dolbeault cohomology of a $G$-equivariant holomorphic vector bundle, called the background cohomology. In this paper, we show that the background cohomology of a prequantum line bundle over $M$ `commutes with reduction', i.e. the invariant part of the background cohomology is isomorphic to the usual Dolbeault cohomology of the symplectic reduction. 
\end{abstract}

\maketitle

\section{Introduction}\label{S:introduction}

If $E$ is a holomorphic vector bundle over a compact \ka manifold $M$, the Dolbeault cohomology $H^{0,*}(M,E)$ is finite-dimensional and has a lot of nice properties. In particular, given a Hamiltonian action of a compact group $G$ on $M$ and a prequantum $G$-equivariant line bundle $L$ over $M$, the $G$-invariant part of the Dolbeault cohomology of $L$, $H^{0,*}(M,L)^G$, is isomorphic to the cohomology of the symplectic reduction, $H^{0,*}(M\sslash G,L\sslash G)$, cf. \cites{Telem00,Br-mum,Zhang99}. This result refines the famous Guillemin-Sternberg `quantization commutes with reduction' conjecture, \cites{GuiSter82,meinrenken98,TianZhang98GS}, to the level of individual cohomology groups.

If $M$ is non-compact,  $H^{0,*}(M,E)$ is an infinite-dimensional space, and much less is known about it. In \cite{BrBackground}, we consider the case when $M$ is a complete \ka manifold,  the moment map $\mu$ is proper, and the vector field generated by $\mu$ on $M$ does not vanish outside of a compact set $K\subset M$. We refer to the pair $(M,\mu)$ as a {\em tamed \ka manifold}. For a $G$-equivariant holomorphic bundle $E$ over $M$ we defined the regularized version of the Dolbeault cohomology  $H^{0,*}_\bg(M,E)$ called the {\em background cohomology}. It is still infinite-dimensional. But as a representation of $G$ it decomposes into a direct sum of irreducible components and each component appears in this decomposition finitely many times:
\begin{equation}\label{E:finiteregIntr}
        H^{0,p}_\bg(M,E)  \ = \ \sum_{V\in \Irr G}\, \bet^p_{\bg,V}\cdot V, \qquad p=0\nek n.
\end{equation}
The alternating sum of the background cohomology is equal to the equivariant index of the pair $(E,\mu)$  constructed by Paradan \cite{Paradan03} and Braverman \cite{Br-index}  (see also \cite{MaZhang_TrIndex12} and \cite{BrCano14}). 

It is shown in \cite{BrBackground} that the background cohomology possesses many properties of Dolbeault cohomology on a compact $G$-manifold. Eg. it satisfies a version of the Kodaira vanishing theorem. There is also an analog of Witten's holomorphic Morse inequalities \cite{Witten84} (see also \cite{WuZhang}).

Let now $L$ be a holomorphic $G$-equivariant {\em prequantum} line bundle over $M$. This means that $L$ is endowed with a $G$-invariant holomorphic Hermitian connection $\n$ such that the \ka form 
\begin{equation}\label{E:chern=omega}
    \omega \ = \ \frac{i}{2\pi}\n^2. 
\end{equation}
Let $\grg$ denote the Lie algebra of $G$. We identify it with its dual $\grg^*$ using a $G$-invariant metric on $\grg$. Then the moment map can be viewed as a map $\mu: M\to \grg$. We assume that $0$ is a regular value of $\mu$ and that the action of $G$ on $\mu^{-1}(0)$ is free. Then, since we assume that $\mu$ is proper, the symplectic reduction $M\sslash G$ is a smooth compact \ka manifold. The prequantum line bundle $L$
induces a prequantum line bundle $L\sslash G$ over $M\sslash G$. The main result of this paper is the following 

\begin{Theorem}\label{T:[Q,R]}
Suppose $M$ is a complete \ka manifold on which a compact group $G$ acts Hamiltonially with proper moment map $\mu:M\to \grg$. Let $L$ be a $G$-equivariant prequantum line bundle over $M$.  Suppose the vector field induced by $\mu$ on $M$ does not vanish outside a compact set $K\subset M$, so that the background cohomology $H_\bg^{0,p}(M,L)$ ($p=0,\ldots,n$) is defined.  Let $H_\bg^{0,p}(M,L)^G$ denote the $G$-invariants in $H_\bg^{0,p}(M,L)$. Then 
\begin{equation}\label{E:[Q,R]}
 \dim H_\bg^{0,p}(M,L)^G\ = \ \dim H_\bg^{0,p}(M\sslash G,L\sslash G).
\end{equation}
\end{Theorem}

When manifold $M$ is compact, the background cohomology is equal to the usual Dolbeault cohomology and the above theorem reduces to the results of \cites{Telem00,Br-mum}.

The assumption that 0 is a regular value of the moment map can be removed by combining the methods of this paper with \cite{Zhang99}.

The definition of the equivariant index of the pair $(E,\mu)$ was extended to the case when a non-compact Lie group acts properly on $M$ in \cite{HochsMathai15} and \cite{Br-indexproper}. It would be interesting to extend the definition of the background cohomology to this case. We will address this problem elsewhere. 

\section{The background cohomology}\label{S:background}

In this section, we recall the construction of the background cohomology from \cite{BrBackground}. First, we define a notion of a tamed \ka $G$-manifold and then construct the background cohomology of holomorphic $G$-equivariant vector bundles over a tamed \ka manifold.

\subsection{A Hamiltonian group action}\label{SS:asHamiltonian}
Suppose a compact Lie group $G$ acts holomorphically on a complete \ka manifold $M$. The moment map $\mu:M\to \grg^*$ is defined by the formula 
\[
	d\<\mu,\bfu\> \ = \ \iot(u)\,\ome,
\]
where $\bfu\in \grg$, $\ome\in \Ome^2(M)$ is the \ka form on $M$ and 
\[
    u(x)\ := \ \frac{d}{dt}_{\big|_{t=0}}\exp(t\bfu)\,x \ \in \ T_xM
\]
is the vector field on $M$ generated by $\bfu$.

Fix an invariant scalar product $\<\cdot,\cdot\>$ on $\grg^*$ and let $|\mu|^2$ denote the square of the norm of $\mu$ with respect to this scalar product. This scalar product defines and isomorphism $\psi:\grg^*\to \grg$ and we denote
\begin{equation}\label{E:bfv(x)}
	\bfv(x) \ := \  \psi\big(\mu(x)\big)\ \in \grg.
\end{equation}
Consider the Hamiltonian vector field with Hamiltonian $|\mu|^2/2$
\begin{equation}\label{E:vfieldv}
	v(x)\ := \ -\,J\,\n\,\frac{|\mu(x)|^2}2
\end{equation}
on $M$ (here $J$ denotes the complex structure on $TM$). One checks that is equal to the vector field generated by $\bfv$:
\begin{equation}\label{E:v=bfv}
	v(x)\ = \ \frac{d}{dt}_{\big|_{t=0}} \exp(t\bfv(x))\,x.
\end{equation}

\subsection{Tamed \ka manifolds}\label{SS:tamedKahler}

\begin{Definition}\label{D:tamedKahler}
A {\em tamed \ka $G$-manifold} is a complete \ka manifold $(M,g^M)$ together with a Hamiltonian action of  the group $G$, such that 
\begin{enumerate}
\item the moment map $\mu$ is proper;
\item the vector field \eqref{E:vfieldv} does not vanish outside of a compact set.
\end{enumerate}
\end{Definition}

\subsection{A rescaling}\label{SS:Kahlerrescaling}
Our definition of the regularized cohomology uses a certain rescaling of the function  $|\mu|^2/2$, which is defined in \eqref{E:deformedp} using a function
\begin{equation}\label{E:phi(x)}
	\phi(x) \ :=  \ s\big(\,|\mu(x)|^2/2\,\big),
\end{equation}
where $s:[0,\infty)\to [0,\infty)$ is a smooth positive strictly increasing function satisfying certain growth conditions at infinity. Roughly speaking, we demand that $\phi(x)$ tends to infinity ``fast enough" when $x$ tends to infinity. The precise conditions we impose on $\phi$ are quite technical, cf. Definition~\ref{D:Kahleradmis} below (see also Definition~4.8 of \cite{BrBackground}), but our construction turns out to be independent of the concrete choice of $\phi$. It is important, however, to know that at least one admissible function exists, cf. Lemma~4.12 of \cite{BrBackground}. 

We now briefly recall the construction of the rescaling from Section~4 of \cite{BrBackground}.

Consider a $G$-equivariant holomorphic vector bundle $E$  over $M$ endowed with a $G$-invariant holomorphic connection $\n^E$. Fix a  Hermitian metric $g^E$ on $E$. Consider the bundle $\E= E\otimes\Lam^\b(T\ah{}M)^*$ and endow it with the Hermitian metric induced by $g^E$ and $g^M$. Let $\n^\E$ denote the connection on $\E$ induced by $\n^E$ and the Levi-Civita connection on $TM$. 

For a vector $\u\in\grg$, we denote by $\L^\E_\u$ the infinitesimal action of $\u$ on $\Gam(M,\E)$ induced by the action on $G$ on $\E$ and by  $\n_{u}^\E:\Gam(M,\E)\to\Gam(M,\E)$ the covariant derivative along the vector field $u$ induced by $\u$. The difference between those two operators is a bundle map, which we denote by
\begin{equation}\label{E:mu}
    \mu^\E(\u) \ := \ \n^\E_{u}-\L^\E_\u \ \in \ \End \E.
\end{equation}

We will use the same notation $|\cdot|$ for the norms on the bundles $TM, T^*M, \E$.  Let $\End(TM)$ and $\End(\E)$ denote the bundles  of endomorphisms of $TM$ and $\E$, respectively. We will denote by $\|\cdot\|$ the norms on these bundles induced by $|\cdot|$. Set
\begin{equation}\label{E:nu}
    \nu=|\v|+\|\nLC v\|+\|\mu^\E(\v)\|+|v|+1.
\end{equation}

\begin{Definition}\label{D:Kahleradmis}
A smooth  function $s:[0,\infty)\to [0,\infty)$ is called {\em admissible for the quadruple $(M,g^M,E,g^E)$} if $s'(r)>0$  and the function 
\begin{equation}\label{E:Kahleradmis}
	f(x)\ :=\ s'\big(|\mu(x)|^2/2\big)
\end{equation} 
satisfies the following condition
\begin{equation}\label{E:admissible}
	        \lim_{M\ni x\to\infty}\, \frac{f^2|v|^2}{    |df||v|+f\nu+1 }  \ = \ \infty.
\end{equation}
 
We denote by $\calF= \calF(M,g^M,E,g^E)$ the set of admissible functions for $(M,g^M,E,g^E)$
\end{Definition}
\begin{Remark}\label{R:h1<h2}
Suppose $h_1^E\le h_2^E$ are two Hermitian metrics on $E$. Then it follows immediately from the definition that $\calF(M,g^M,E,g^E_2)\subseteq \calF(M,g^M,E,g^E_1)$.
\end{Remark}

The following lemma summarizes Lemmas~4.11 and 4.12 of \cite{BrBackground}:
\begin{Lemma}\label{L:admissible2}
Given  a holomorphic Hermitian $G$-equivariant vector bundle $(E,g^E)$ over a tamed \ka manifold $(M,g^M)$, the set $\calF= \calF(M,g^M,E,g^E)$ of admissible functions is not empty. 

If $s_1, s_2\in \calF(M,g^M,E,g^E)$, then for any  positive real numbers  $t_1,t_2>0$ the function $s:=t_1s_1+t_2s_2$ is admissible. Thus, the set $\calF= \calF(M,g^M,E,g^E)$ of admissible functions is a convex cone. 
\end{Lemma}

\subsection{The deformed Dolbeault cohomology}\label{SS:Dolbeaultcohomology}
Let $s:[0,\infty)\to [0,\infty)$ be an admissible function, cf. Definition~\ref{D:Kahleradmis}, and set
\[
	\phi(x) \ := \ s\big(\,|\mu(x)|^2/2\,\big), \qquad x\in M.
\]
Consider the deformed Dolbeault differential
\begin{equation}\label{E:deformedp}
	\p_s \alp\ = \ e^{-\phi}\circ\p\circ e^\phi\,\alp 
	\ = \ \p\alp\ + \ f\,\p\big(\,|\mu|^2/2\,\big)\wedge\alp,
\end{equation}
where, as in Definition~\ref{D:Kahleradmis}, $f= s'(|\mu|^2/2)$.

The {\em  deformed Dolbeault-Dirac operator} is defined by 
\begin{equation}\label{E:Ds}
	D_{s} \ = \ \sqrt2\,\big(\,\p_s \, + \, \p_s^*\,\big).
\end{equation}

Let $L_2\Ome^{0,p}(M,E)$ denote the space of square-integ\-rable differential $(0,p)$-forms on $M$ with values in $E$. 
Define the {\em deformed Dolbeault cohomology $H^{0,*}_s(M,E)$} of the triple $(M,E,s)$ as the reduced cohomology of the deformed differential $\p_s$:
\begin{equation}\label{E:deformedDolbeault}
	H^{0,p}_s(M,E) \ = \ \frac{
	\Ker\,\big(\,\p_s: L_2\Ome^{0,p}(M,E)\to L_2\Ome^{0,p+1}(M,E)\,\big)}
	{ \overline{\IM \big(\, \p_s: L_2\Ome^{0,p-1}(M,E)\to L_2\Ome^{0,p}(M,E)\,\big)}}.   
\end{equation}

The following is Theorem~5.3 of \cite{BrBackground}:
\begin{Theorem}\label{T:Cohfinite}
Suppose $s\in \calF(M,g^M,E,g^E)$ is an admissible function.  Then the deformed Dolbeault cohomology $H^{0,p}_s(M,E)$ decomposes, as a Hilbert space, into an infinite direct sum
\begin{equation}\label{E:finitecoh}
        H^{0,p}_s(M,E)  \ = \ \sum_{V\in \Irr G}\, \bet^p_{s,V}\cdot V,
\end{equation}
where  $\Irr G$ is the set of irreducible representations of $G$, and $\bet^p_{s,V}$ are non-negative integers, which depend on the choice of the admissible function $s$. In other words, each irreducible representation of $G$ appears in $H^{0,p}_s(M,E)$ with finite multiplicity.
\end{Theorem}

For each irreducible representation $V$ the finite-dimensional representation 
\[
	H^{0,p}_{s,V}(M,E)  \ = \ \bet^p_{s,V}\cdot V
\]
is called the {\em $V$-component} of the deformed cohomology. Then 
\[
        H^{0,p}_s(M,E)  \ = \ \sum_{V\in \Irr G}\, H^{0,p}_{s,V}(M,E).
\]

\subsection{The background cohomology}\label{SS:background}

\begin{Definition}\label{D:background}
The minimal possible value of $\bet^p_{s,V}$  is called the {\em background Betti number} and is denoted by $\bet^p_{\bg,V}$:
\begin{equation}\label{E:backgroundBetti}
	\bet^p_{\bg,V}\ := \ \min\big\{\, \bet^p_{s',V}:\, s'\in \calF\,\big\}
\end{equation}

An admissible  function $s$ is called {\em $V$-generic} if $\bet^p_{s,V}\ = \  \bet^p_{\bg,V}$ for all $p=0\nek n$. In a certain sense, ``almost all" admissible functions are $V$-generic. 
\end{Definition}

It is proven in \cite{BrBackground} that the background Betti numbers are independent of the choices made. In particular, \cite[Theorem~5.7]{BrBackground}, we have the following
\begin{Theorem}\label{T:independenceofs}
Let $V\in \Irr G$ be an irreducible representation of $G$. For any two $V$-generic admissible functions $s_1,s_2\in \calF(M,g^M,E,g^E)$ there exists a canonical isomorphism 
\begin{equation}\label{E:Phis1s2}
	\Phi_{s_1s_2}^V:\, H^{0,*}_{s_1,V}(M,E) \ \longrightarrow \ H^{0,*}_{s_2,V}(M,E),
\end{equation}
satisfying the cocycle condition
\begin{equation}\label{E:cocyclePhi}
	\Phi_{s_2s_3}^V\circ\Phi_{s_1s_2}^V \ = \ \Phi_{s_1s_3}^V.
\end{equation}

If $\tau =s_{2}-s_{1}\ge0$ is an admissible function then the isomorphism $\Phi_{s_{1}s_{2}}^V$ is induced by the map
\begin{equation}\label{E:alptoealp}
	\alp \ \mapsto \ e^{-\tau(|\mu|^2/2)}\,\alp,\, \qquad \ \alp\in L_2\Ome^{0,*}(M,E).
\end{equation}
\end{Theorem}

Theorem~5.9 of \cite{BrBackground} states that the background Betti numbers are also independent of the choice of the Hermitian metric $g^E$. This justifies the following 
\begin{Definition}\label{D:backgroundcohomology}
If\/ $V\in \Irr{}G$, the {\em $V$-component of the background  cohomology} $H^{0,p}_{\bg,V}(M,E)$  of the triple  $(M,g^M,E) $ is defined to be the deformed cohomology $H^{0,*}_{s,V}(M,E)$ for any Hermitian metric $g^E$ and any $V$-generic function $s$. 

The {\em background cohomology} $H^{0,p}_\bg(M,E)$ is the direct sum
\begin{equation}\label{E:background=oplus}
	H^{0,p}_\bg(M,E) \ := \ \sum_{V\in \Irr G} H^{0,p}_{\bg,V}(M,E).
\end{equation}
\end{Definition}

\begin{Remark}
Theorem~5.11 of \cite{BrBackground} states that, in an appropriate sense, the alternating sum of the background cohomology is equal to the equivariant index of a tamed Dirac operator $(\bar\p+\bar\p^*,\bfv)$ as defined in \cite{Br-index}.
\end{Remark}

\section{Proof of Theorem~\ref{T:[Q,R]}}\label{S:proof[Q,r]}

In this section, we prove Theorem~\ref{T:[Q,R]}. 

\subsection{A prequantum line bundle}\label{SS:prequnatum}
Let $(M,g^M)$ be a tamed \ka $G$-manifold, cf. Definition~\ref{D:tamedKahler}. A $G$-equivariant {\em prequantum} line bundle $L$ over $M$ is a holomorphic $G$-equivariant line bundle endowed with $G$-invariant holomorphic Hermitian metric $g^L$ and a $G$-equivariant holomorphic Hermitian connection $\n$ satisfying \eqref{E:chern=omega}. Let $\mu=\mu^L$ be defined by \eqref{E:mu}. Then $\mu$ is the moment map for the action of $G$ on $M$, cf. \cite{Kostant70}. The tameness of $(M,g^M)$ implies that $\mu$ is proper. We make a basic 
\begin{Assumption}\label{A:freeaction}
Assume that $0$ is a regular value of $\mu:M\to \grg^*$ and that $G$ acts freely on $\mu^{-1}(0)$.
\end{Assumption}
Then the {\em symplectic reduction}
\[
    M\sslash G\ :=  \ \mu^{-1}(0)/G
\]
is a smooth compact manifold and 
\[
    L\sslash G\ := \ L|_{\mu^{-1}(0)}/G
\]
is a holomorphic prequantum line bundle over $M\sslash G$. To simplify the notation we set 
\begin{equation}\label{E:X0L0}
    M_0:= M\sslash G, \qquad L_0\ := \ L\sslash G. 
\end{equation}

\subsection{Symplectic quotients of \ka manifolds}\label{SS:kaquotient}
We briefly recall how the complexification of $G$ acts on $M$. See \cite[\S11]{Kirwan84-book} for details.

Let $G^\CC$ denote the natural complexification of the compact group $G$ such that
\[
    G\times\grg\ \to \ G^\CC, \qquad (g,\eta)\ \mapsto \ \exp(i\eta)\,g
\]
is a diffeomorphism. The action of $G$ on $M$ extends to the action of $G^\CC$ and our assumptions apply that $G^\CC$ acts freely on the set 
\[
    M^{ss}\ := \ G^\CC\cdot\mu^{-1}(0). 
\]
The set $M^{ss}$ coincides with the set of semi-stable points for $G^\CC$ action, \cite{MFK94}. This is a dense subset of $M$ and $M_0= M^{ss}/G^\CC$. We denote by $q$ the quotient map
\begin{equation}\label{E:qMsstoM0}
    q:\, M^{ss}\ \to M^{ss}/G^\CC\ = \ M_0.
\end{equation}

\subsection{Harmonic forms on the symplectic reduction} \label{SS:harmonicM0}
Recall that $\mu^{-1}(0)$ is a smooth compact manifold and that $G$ acts freely on it. Let $\oq:\mu^{-1}(0)\to M_0=\mu^{-1}(0)/G$
denote the quotient map and set
\[
      \tilh(x) \ = \ \sqrt{\Vol\, \oq^{-1}(x)}.
\]
Let $g^{L_0}$ and $g^{M_0}$ denote the Hermitian metric on $L_0$
and the Riemannian metric on $M_0$ induced by the fixed metrics on
$L$ and $M$ respectively. Set
$g^{L_0}_{\tilh}=\tilh^2g^{L_0}$ and let  $\p^*_\tilh$ denote the
formal adjoint of the Dolbeault differential $\p:\Ome^{0,*}(M_0,
L_0)\to \Ome^{0,*+1}(M_0, L_0)$ with respect to the metrics
$g^{L_0}_{\tilh}, g^{M_0}$. 

From this point on we always consider this scalar product on $\Ome^{0,*}(M_0, L_0)$  and define the space of square-integrable differential forms $L_2\Ome^{0,*}(M_0, L_0)$ using this scalar product.  Let
\[
        \h{j} \ = \ \Ker\Big(\,
                \p\p^*_\tilh+\p^*_\tilh\p\,\big):\, \Ome^{0,j}(M_0, L_0)
                \ \to \ \Ome^{0,j}(M_0, L_0,\Big)
\]
be the space of harmonic forms with respect to this scalar product. As a representation of $G$, $\h{j}$ is isomorphic to the Dolbeault cohomology $H^{0,j}(M_0,L_0)$.

\subsection{A neighborhood of $\mu^{-1}(0)$} \label{SS:Uepsilon}
Let $N\to\mu^{-1}(0)$ denote the normal bundle to $\mu^{-1}(0)$ in
$M$. If $x\in\mu^{-1}(0),\ Y\in N_x$, let $t\in\RR\to
y_t=\exp_x(tY)\in M$ be the geodesic in $M$ which is such that
$y_0=x, \ dy/dt|_{t=0}=Y$. For $0<\eps<+\infty$, set
\[
        B_\eps \ = \ \{\, Y\in N: \ |Y|<\eps\, \}.
\]
Since $\mu^{-1}(0)$ is compact, there exists $\eps_0>0$ such that, for $0<\eps<\eps_0$, the map $(x,Y)\in N \to\exp_x(tY)$
is a diffeomorphism from $B_{\eps}$ to a tubular neighborhood
${\calU}_\eps$ of $\mu^{-1}(0)$ in $M$. We identify $B_\eps$ with $\calU_\eps$ and use the notation $y=(x, Y)$ instead of $y=\exp_x(Y)$.

\subsection{The $L_2$-pairing} \label{SS:L2pairing}
Let $\<\cdot,\cdot\>_{L}$ denote the paring
\begin{equation}\label{E:pairingonL}
    \<\cdot,\cdot\>_{L}:\, \Omega^{0,j}(M,L)\times \Omega^{0,k}(M,L) \ \to \ 
    \Omega^{0,j+k}(M)
\end{equation}
defined by the Hermitian metric $g^L$ on $L$. Similarly we define the pairing $\<\cdot,\cdot\>_{L_0}$ on $\Omega^{0,j}(M_0,L_0)$ defined by the Hermitian metric $g^{L_0}_{\tilh}$. 

\subsection{A choice of an admissible function} \label{SS:choiceadmissible}
Let 
\[
    C \ := \ \max \big\{\,\frac{|\mu(x)|^2}2:\, x\in \calU_\epsilon\,\big\}.
\]
For $\beta\in L_2\Omega^{0,*}(M_0,L_0)$ we denote by $q^*\beta$ its pull-back to a bounded differential form on $M^{ss}$ with values in $L^{ss}$. If $s$ is an admissible function that grows fast enough at infinity, then, for any $\beta\in L_2\Omega^{0,*}(M_0,L_0)$, the form $e^{-s(\frac{|\mu|^2}2)} q^*\beta$ on $M^{ss}$ is square-integrable. Since $M^{ss}$ is dense in $M$ we can view this form as a square-integrable form on $M$.

Let $V_0=\CC$ denote the trivial representation of $G$ in $\CC$. Then for any admissible function $s$, the cohomology $H^{0,*}_{s,V_0}(M,L)$ is equal to the $G$-invariant part $H^{0,*}_{s}(M,L)^G$ of the deformed cohomology. 

We choose a $V_0$-generic admissible function $s$ such that 
\begin{enumerate}
    \item   $s(\tau)=\tau$ for $\tau\le C$;
    \item $s$ grows fast enough at infinity so that for any $\alpha\in L_2\Omega^{0,*}(M,L)$  and any $\beta\in L_2\Omega^{0,*}(M_0,L_0)$ the integral 
\[
    \int_{M^{ss}}\,\big\<\,e^{-s(\frac{|\mu|^2}2)}q^*\beta, \,\alp\,\big\>_L
    \wedge\ome^r
\]
converges and the linear map 
\begin{equation}\label{E:pushforward}
    \alpha \ \mapsto \  \int_{M^{ss}}\,\big\<\,e^{-s(\frac{|\mu|^2}2)}q^*\beta,\alpha\,\big\>_L
    \wedge\ome^r
\end{equation}
is bounded
\end{enumerate}

We note that for every $t\ge1$ the function $ts$ is also $V_0$ admissible. Thus we can consider the deformed cohomology $H^{0,*}_{ts,V_0}(M,L)$. For $t=1$ this is the background cohomology since we assumed that $s$ is $V_0$-generic. Using the semi-continuity property of the dimensions of the kernels of operators one can show that for almost all $t\ge 1$ the function $ts$ is also $V_0$-generic. Moreover, we show below that there exists $t_0\ge 1$ such that for all $t\ge t_0$ the function $ts$ is $V_0$-generic. Thus for each $t\ge t_0$, 
\[  
    H^{0,*}_{ts}(M,L)^G\  = \ H^{0,*}_{ts,V_0}(M,L)\ = \ H^{0,*}_{\bg,V_0}(M,L)
    \ = \ H^{0,*}_{\bg}(M,L)^G.
\]

\subsection{The Tian-Zhang isometry} \label{SS:TianZhangisometry}
Let $D_{ts}$ be the operator defined in \eqref{E:Ds} (with $s$ replaced by $ts$)
and set 
\begin{equation}\label{E:boxLaplacian}
    \lap\ := \ D_{ts}^2.
\end{equation}
Let $\lap^G$ denote the restriction of $\lap$ to the space of $G$-invariant forms. 
For $\lam>0$, \ $j=0,1,\dots$, and $t>1$, we define
$E_{\lam,t}^{j,G}$ to be the span of the eigenforms of
$\lap^{j,G}$ with eigenvalues less or equal than $\lam$. Then
$E^{*,G}_{\lam,t}$ is a subcomplex of $(\Ome^{0,*}(X, L^G),\p_{ts})$.
Since $E^{*,G}_{\lam,t}$ contains the kernel of $\lap$, it
follows from the Hodge theory that the cohomology of
$(E^{*,G}_{\lam,t},\p_{ts})$ is isomorphic to $H^{0,*}(M,L)^G$.

Let $\calH^{0,*}(M_0,L_0)\subset \Omega^{0,*}(M_0,L_0)$ denote the space of harmonic forms with respect to the scalar product defined in Section~\ref{SS:Uepsilon}. This space is isomorphic to $H^{0,*}(M_0,L_0)$. 

The following analog of a result of Tian and Zang \cite{TianZhang98GS} is proven in \cite[Theorem~2.1]{BrBackground} in the case when $M$ is compact.

\begin{Proposition}\label{P:TianZhangisomorphism}
There exist $\lam>0$ and $t_0\ge1$ such that, for all $j=0,1,\ldots$ and
all $t>t_0$, \, $\lam$ is not in the spectrum of \/ $\lap^{j,K}$
and
\begin{equation}\label{E:TZ1}
        \dim E^{j,G}_{\lam,t} \ = \ \dim H^{0,j}(M_0,L_0).
\end{equation}
There exist explicitly constructed linear bijections 
\begin{equation}\label{E:TZ1bijection}
        \Phi^j_{\lambda,t}: \calH^{0,j}(M_0,L_0)\ \to \ E^{j,G}_{\lam,t},
\end{equation}
such that 
\begin{equation}\label{E:almostisometry}
    \lim_{t\to\infty}\, \big\| (\Phi^j_{\lambda,t})^*\Phi^j_{\lambda,t}-1\big\|
    \ = \ 0.
\end{equation}
\end{Proposition}
The proof in \cite{BrBackground} is completely local. One only studies the differential forms in the neighborhood $\calU_\epsilon\subset M$ of $\mu^{-1}(0)$. This proof works verbatim in our new situation.

\subsection{The plan of the proof of Theorem~\ref{T:[Q,R]}}\label{SS:plan}
Corollary~\ref{C:pvanishes} below states that the restriction of $\p_{ts}$
to $E^{*,K}_{\lam,t}$ vanishes for $t\gg1$. In other words, the
cohomology of the complex $(E^{*,G}_{\lam,t},\p_{ts)}$ is isomorphic
to $E^{*,G}_{\lam,t}$. Since this cohomology is equal to the deformed cohomology $H^{0,*}_{ts}(M,L)^G$ it follows  form Proposition~\ref{P:TianZhangisomorphism} that 
\[
    \dim H^{0,*}_{ts}(M,L)^G \ = \ \dim H^{0,j}(M_0,L_0),
\]
for $t\gg1$. In particular, the left-hand side is independent of $t$. It follows that for all $t\gg1$ the function $ts$ is $V_0$-generic and the deformed cohomology is isomorphic to the background cohomology $H^{0,*}_{\bg}(M,L)^G$. This proves Theorem~\ref{T:[Q,R]}.

\subsection{The push-forward map}\label{SS:admissible4integration}
In this subsection, for each $t\ge 1$ we define a push-forward (or the integration) map from the complex  $\big(L_2\Omega^{0,*}(M,L),\p_{ts}\big)$ to $(L_2\Omega^{0,*}(M_0,L_0),\p\big)$. Later we prove that the restriction of this map to the invariant forms $L_2\Omega^{0,*}(M,L)^G$ defines a quasi-isomorphism of this complex with $(L_2\Omega^{0,*}(M_0,L_0),\p\big)$. This will prove Theorem~\ref{T:[Q,R]}.

Set $r=\dim G$. Recall that $q:M^{ss}\to M^{ss}/G^\CC=M_0$ is the projection, cf. \eqref{E:qMsstoM0}. 
For any $t\ge 1$ we define the {\em push-forward map}
\begin{equation}\label{E:It}
    I_t:\, L_2\Omega^{0,*}(M,L)\ \to \ L_2\Omega^{0,*}(M_0,L_0)
\end{equation}
by setting $I_t\alpha$ to be the unique square-integrable form on $M_0$ such that for any $\beta\in L_2\Omega^{0,*}(M_0,L_0)$
\begin{equation}\label{E:defofI}
     \int_{M_0}\,\<\beta,I_t\,\alpha\>_{L_0}
    \ := \ \Big(\frac{t}{2\pi}\Big)^{r/4}\, \int_{M^{ss}}\,\big\<\,e^{-ts(\frac{|\mu|^2}2)}q^*\beta,\alpha\,\big\>_L
    \wedge\ome^r.
\end{equation}

\begin{Proposition}\label{P:commutewithd}
$\p\circ I_t = I_t\circ\p_t$, for any $t\ge1$.
\end{Proposition}
We postpone the proof of the proposition to Section~\ref{SS:prcommutewithd}. We finish this section by showing how this proposition implies Theorem~\ref{T:[Q,R]}.

\begin{Proposition}\label{P:comparisonwithPhi}
Let $i:\h{*}^G\to \Ome^{0,*}(M_0,L_0)^G$ denote the inclusion and
let $\|I_t\circ\Phi_{\lam,t}^*-i\|$ denote the norm of the
operator $I_t\circ\Phi_{\lam,t}^*-i:\h{*}^G\to\Ome^{0,*}(M_0,L_0)^G$.
Then
\[
        \lim_{t\to\infty}\,
                \|I_t\circ\Phi_{\lam,t}^*-i\| \ = \ 0.
\]
\end{Proposition}
\begin{proof}
By construction the norm of the restriction of $\Phi^{*}_{\lambda,t}\alpha$ to 
$M\backslash \calU_\epsilon$ decays rapidly as $t\to\infty$. Thus the computation of the norm of $I_t\circ\Phi_{\lam,t}^*-i$ reduces to a computation on $\calU_\epsilon$. A verbatim repetition of the arguments in Section~4.2 of \cite{BrBackground} proves the proposition. 
\end{proof}

Recall that a positive number $t_0$ was defined in Proposition~\ref{P:TianZhangisomorphism}. By Proposition~\ref{P:comparisonwithPhi} for large enough $t$, we have $\|I_t\circ\Phi_{\lam,t}^*-i\|<1$. This implies the following

\begin{Corollary}\label{C:pvanishes}
Suppose $t>t_0$ is large enough, so that
$\|I_t\circ\Phi_{\lam,t}^*-i\|<1$. Then the restriction of $\p_{ts}$
to $E_{\lam,t}^{*,G}$ vanishes.
\end{Corollary}

\begin{proof}
Let $t$ be as in the statement of the corollary and let
$\gam\in{}E_{\lam,t}^{*,G}$. Then, $\p_t\gam\in E_{\lam,t}^{*,G}$.
Hence, it follows from Proposition~\ref{P:TianZhangisomorphism}, that there exists
$\alp\in\h{*}$ such that $\Phi_{\lam,t}^*\alp=\p_t\gam$.

By Proposition~\ref{P:commutewithd}, $I_t\p_t\gam=\p I_t\gam\in
\Ome^{0,*}(M_0,L_0)$. Hence, this vector is orthogonal to the subspace $\h{*}$. Hence,
\[
        \|\alp\| \ \le \
          \|I_t\p_t\gam-\alp\| \ = \ \|I_t\Phi_{\lam,t}^*\alp-\alp\|
          \ \le \ \|I_t\Phi_{\lam,t}^*-i\|\cdot\|\alpha\|
\]
Since, $\|I_t\circ\Phi_{\lam,t}^*-i\|<1$, it follows that
$\alp=0$. Hence, $\p_t\gam=\Phi_{\lam,t}^*\alp=0$.
\end{proof}

\subsection{Proof of Theorem~\ref{T:[Q,R]}}
Theorem~\ref{T:[Q,R]} follows now from the arguments given in Section~\ref{SS:plan} .\hfill$\square$

\section{The proof of Proposition~\ref{P:commutewithd}}\label{SS:prcommutewithd}

Proposition~\ref{P:commutewithd} is equivalent to the statement that for any $\beta\in L_2\Omega^{0,j}(M_0,L_0)$, \ $\alpha\in L_2\Omega^{0,k}(M,L)$
\[
    \p\,\big<\,\beta,I_t\alpha\,\big\> \ = \ \big<\,\p\beta,I_t\alpha\,\big\>
    \ + \
    (-1)^j\,\big<\,\beta,I_t\p\alpha\,\big\>.
\]
By \eqref{E:defofI}, this is equivalent to 
\begin{multline}\label{E:integrationbypart}   
    \p\, \int_{M^{ss}}\,  \big\<\,e^{-ts (\frac{|\mu|^2}2)} q^*\beta,\alp\big\>_L \wedge\ome^r
\\ = \
      \int_{M^{ss}}\,  \big\<\,e^{-ts (\frac{|\mu|^2}2)} q^*\p\beta,\p\alp\big\>_L \wedge\ome^r
\ + \ (-1)^j\,
    \int_{M^{ss}}\,  \big\<\,e^{-ts (\frac{|\mu|^2}2)} q^*\beta,\p\alp\big\>_L \wedge\ome^r
\end{multline}
The proof of \eqref{E:integrationbypart} is based on Lemma~\ref{L:mu=0}, for which we need to introduce some additional notation.

\subsection{A vanishing lemma}
Let $h_1\nek h_r$ denote an orthonormal basis of $\grg^*$. Denote by
$v_i$ the Killing vector fields on $M$ induced by the duals of
$h_i$. Then $\mu=\sum\,{}\mu_ih_i$, where each $\mu_i$ is a real
valued function on $X$ such that $\iot_{v_i}\ome=d\mu_i$. Hence,
\begin{equation}\label{E:dmu^2}
        d\, \frac{|\mu|^2}2 \ = \ \sum_{i=1}^r\, \mu_i d\mu_i
                        \ = \  \sum_{i=1}^r\, \mu_i \iot_{v_i}\ome.
\end{equation}
Then 
\[
        \p_{ts}:\, \alp \ \mapsto e^{-ts(\frac{|\mu|^2}2)}\,
                         \p\, e^{ts(\frac{|\mu|^2}2)}\, \alp
                \ = \ \p\alp \ + \ t\,s'\big(\frac{|\mu|^2}2\big)\,
                \sum_{i=1}^k\, \mu_i \p\mu_i\wedge\alp.
\]

\begin{Lemma}\label{L:mu=0}
Let $\beta\in L_2\Omega^{0,*}(M_0,L_0)$. Then
\[
        \int_{M^{ss}}\,  \mu_i\p\mu_i\wedge\big\<\,e^{-ts (\frac{|\mu|^2}2)} q^*\beta,\alp\big\>_L \wedge\ome^r
             \ = \  0,
\]
for any $i,j=0,1,\ldots$ and any  $\alp\in\Ome^{0,j}(M,L)$.
\end{Lemma}
\begin{proof}
Let
$\Pi_{0,*}:\Ome^{*,*}(M_0,L_0)\to\Ome^{0,*}(M_0,L_0)$ denote the
projection. Then
\begin{equation}\label{E:pmu=dmu}
        \int_{M^{ss}}\,  \mu_i\p\mu_i\wedge\big\<\,e^{-ts (\frac{|\mu|^2}2)} q^*\beta,\alp\big\>_L
    \wedge\ome^r
            \ = \
                \Pi_{0,j+1}\, \int_{M^{ss}}\,  \mu_id\mu_i\wedge\big\<\,e^{-ts (\frac{|\mu|^2}2)} q^*\beta,\alp\big\>_L
    \wedge\ome^r.
\end{equation}
Using \refe{dmu^2}, we obtain
\begin{multline}\label{E:dmu=}
        \int_{M^{ss}}\,  \mu_i d\mu_i\wedge\big\<\,e^{-ts (\frac{|\mu|^2}2)} q^*\beta,\alp\big\>_L
    \wedge\ome^r
             \\ = \
    \int_{M^{ss}}\,  \mu_i \iota_{v_i}\omega\wedge\omega^r\wedge\big\<\,e^{-ts (\frac{|\mu|^2}2)} q^*\beta,\alp\big\>_L 
    \\ = \
                \frac1{r+1}\, \int_{M^{ss}}\,  \mu_i\,\iota_{v_i}\,
                \Big(\,  \omega^{r+1}\wedge\big\<\,e^{-ts (\frac{|\mu|^2}2)} q^*\beta,\alp\big\>_L\,\Big)\,                      
             \\ - \
              \frac1{r+1}\, \int_{M^{ss}}\,  \mu_i\,\omega^{r+1}\wedge \iota_{v_i}
                \Big(\,  \big\<\,e^{-ts (\frac{|\mu|^2}2)} q^*\beta,\alp\big\>_L\,\Big).
\end{multline}
The first summand in \refe{dmu=} vanishes since it does not contain forms of the top degree. The integrand in the second summand
belongs to $\Ome^{r+1,r+*}(M,L)$. It follows, that
\[
        \int_{M^{ss}}\,  \mu_i d\mu_i\wedge\big\<\,e^{-ts (\frac{|\mu|^2}2)} q^*\beta,\alp\big\>_L
    \wedge\ome^r
         \ \in \ \Ome^{1,*}(M_0, L_0).
\]
The lemma follows now from \refe{pmu=dmu}.
\end{proof}

\subsection{Proof of \eqref{E:integrationbypart}}
Since $M^{ss}$ is dense in $M$, the space of differential forms compactly supported in $M^{ss}$ is dense in $L_2\Omega^{0,*}(M,L)$. Hence, we can assume that the support of  $\alp$ is contained in $M^{ss}$. Then the difference between the left and the right-hand sides of \eqref{E:integrationbypart} is
\[
    \int_{M^{ss}}\,  \big\<\,\p\big(e^{-ts (\frac{|\mu|^2}2)}\big) q^*\beta,\alp\big\>_L \wedge\ome^r
    \ = \ 
    -t\,\sum_{i=1}^r\,
    \int_{M^{ss}}\,  \mu_i\p\mu_i\wedge\big\<\,e^{-ts (\frac{|\mu|^2}2)} q^*\beta,s' (\frac{|\mu|^2}2)\,\alp\big\>_L
\]
From Lemma~\ref{L:mu=0} (with $\alpha$ replaced with $s'(\frac{|\mu|^2}2)\,\alp$ we conclude that the last expression is equal to 0. This completes the proof of \eqref{E:integrationbypart} and of Proposition~\ref{P:commutewithd}. \hfill$\square$


\end{document}